\documentclass[A4paper]{article}%
\usepackage{graphicx}
\usepackage{amsmath,amssymb,amsfonts}
\usepackage{theorem}
\usepackage{color}
\usepackage{hyperref}
\usepackage[font={small, it}]{caption}
\usepackage[all]{xy}%
\usepackage{amsmath}%
\usepackage{amsfonts}%
\usepackage{amssymb}
\providecommand{\U}[1]{\protect \rule{.1in}{.1in}}
\theoremstyle{change}
\newtheorem{definition}{Definition:}[section]
\newtheorem{proposition}[definition]{Proposition:}
\newtheorem{theorem}[definition]{Theorem:}
\newtheorem{lemma}[definition]{Lemma:}

{\theorembodyfont{\rmfamily}
	\newtheorem{remark}[definition]{Remark:}
}
{\theorembodyfont{\rmfamily}
	\newtheorem{example}[definition]{Example:}
}
\newenvironment{proof}
{{\bf Proof:}}
{\qquad \hspace*{\fill} $\Box$}

\newcommand{\fg}{\mathfrak{g}}

\newcommand{\fr}{\mathfrak{r}}

\newcommand{\fn}{\mathfrak{n}}

\newcommand{\fh}{\mathfrak{h}}
\newcommand{\fu}{\mathfrak{u}}

\newcommand{\Ad}{\operatorname{Ad}}
\newcommand{\ad}{\operatorname{ad}}

\newcommand{\id}{\operatorname{id}}
\newcommand{\inner}{\operatorname{int}}
\newcommand{\cl}{\operatorname{cl}}

\newcommand{\rme}{\mathrm{e}}

\newcommand{\CC}{\mathcal{C}}

\newcommand{\OC}{\mathcal{O}}

\newcommand{\UC}{\mathcal{U}}

\newcommand{\XC}{\mathcal{X}}
\newcommand{\DC}{\mathcal{D}}

\newcommand{\C}{\mathbb{C}}

\newcommand{\N}{\mathbb{N}}

\newcommand{\R}{\mathbb{R}}
\newcommand{\Z}{\mathbb{Z}}
\begin{document}

\title{Central periodic points of linear systems}
\author{V\'{\i}ctor Ayala \thanks{%
		Supported by Proyecto Fondecyt $n^{o}$ 1190142, Conicyt, Chile} \\
	Universidad de Tarapac\'a\\
	Instituto de Alta Investigaci\'on\\
	Casilla 7D, Arica, Chile\\
	and\\
	Adriano Da Silva \thanks{%
		Supported by Fapesp grant n%
		${{}^o}$
		2018/10696-6.}\\
	Instituto de Matem\'atica,\\
	Universidade Estadual de Campinas\\
	Cx. Postal 6065, 13.081-970 Campinas-SP, Brasil.\\
}
\date{\today }
\maketitle

\begin{abstract}
	In this paper, we introduce the concept of central periodic points of a linear system as points which lies on orbits starting and ending at the central subgroup of the system. We show that this set is bounded if and only if the central subgroup is compact. Moreover, if the system admits a control set containing the identity element of $G$ then, the set of central periodic points, coincides with its interior. 
\end{abstract}

	\textbf{Key words:} linear systems, periodic points, control sets
	
	\textbf{2010 Mathematics Subject Classification: 93B05, 93C05}

\section{Introduction}

Essentially, a linear system on a connected Lie group is an
affine control system whose drift is linear and the control vectors left-invariant ones. Its importance is highlighted by at least two facts: Firstly, it appears as a natural generalization of the classical linear systems on Euclidean spaces. One of the properties that this generalization inherited from the Euclidean case is the possibility to associate subgroups closely connected with its dynamics (see \cite{ADS, DSAyGZ, DS}), called stable, central, and unstable subgroups. Secondly, in \cite{JPh1} Jouan proves the Equivalence Theorem which assures that any affine control system on a connected manifold, whose vector fields are complete and generate a finite dimensional Lie algebra is diffeomorphic to a linear system on a Lie group, or on a homogeneous spaces, showing that linear systems are also relevant for classiffication of general affine control systems on abstract connected manifolds.

On the other hand, like singularities, periodic orbits are essential to understand the dynamics of vector fields. Dynamical systems may have stable limit sets determined by fixed points or periodic orbits, defining the domain of attraction on the manifold, i.e., points from which the trajectories will converge to the corresponding limit set as time goes to infinity. Therefore, to understand the dynamic behavior of a linear system we introduce the
notion of $F$-periodic point of the system as follows: Given a nonempty subset
$F$ of $G$ we say that a point is $F$-periodic if it belongs to a trajectory of the system starting and 
finishing in $F$. The central periodic points are then the $G^0$-periodic points, where $G^0$ stands for the central subgroup associated with the linear system.  Our main result shows that compactness of $G^0$ is a necessary and suficient condition for the boundedness of the central periodic points. As a consequence, the control set containing the identity element of $G$ is
bounded if and only if $G^{0}$ is a compact subgroup. 

The paper is structured as follows: In Section 2 we introduce the concept of linear vector fields and the decompositions induced by them on the group and algebra level. We also introduce the concept of a linear system and its $F$-periodic points, and prove some complementary results. In Section 3 we analyze a particular case of a linear system on a semi-direct product of a connected Lie group with a nilpotent, simply connected, connected Lie group. This particular case is the key to proving our main result and is also essential by itself. Actually, the results in Section 3 gives a way to decompose in coordinates, linear systems on simply connected nilpotent Lie groups. Section 4 contains the proof of the main result of the paper. In this section, we also introduce the concept of control sets and show that the central periodic points of a linear system coincide with the interior of the control set containing the identity of the group. We finish the section with some examples.

\subsection*{Notations}

Let $G$ be a connected Lie group with Lie algebra $\fg$. By $\exp:\fg\rightarrow G$ we denote the exponential map of $G$. For any element $g\in G$ we denote by $L_g$ and $R_g$ the left and right translations of $G$ by $g$, respectively. The conjugation $C_g$ is by definition $C_g=R_{g^{-1}}\circ L_g$.  We denote by $\mathrm{Aut}(G)$ and by $\mathrm{Aut}(\fg)$ the set of automorphisms of $G$ and $\fg$, respectively. The adjoint map $\Ad:G\rightarrow \mathrm{Aut}(\fg)$ is the map defined by $\Ad(g):=(dC_g)_e$, where $e\in G$ stands for the identity element of $G$. If $H$ is a connected Lie group and $\rho:G\rightarrow\mathrm{Aut}(H)$ is a homomorphism, the semi-direct product of $G$ and $H$ is the Lie group $G\times_{\rho} H$ whose subjacent manifold is $G\times H$ and the product is given by 
$$(g_1, h_1)(g_2, h_2):=(g_1g_2, h_1\rho(g_1)h_2).$$
Its Lie algebra coincides, as a vector space, to the Cartesian product $\fg\times\fh$. For any $X, Y\in\fg$, the {\it Baker-Campbell-Hausdorff (BCH)} formula is given by
$$\exp(X)\exp(Y)=\exp(c(X, Y)),$$
where $c(X, Y)$ is a series depending on $X, Y$ and its brackets. Its first terms are given by
$$c(X, Y)=X+Y+\frac{1}{2}[X, Y]+\frac{1}{12}\left([X, [X, Y]]+[Y, [Y, X]]\right)-\frac{1}{24}[Y, [X, [X, Y]]]+\cdots,$$
where the subsequent ones depend on the brackets of five or more elements. The serie is convergent for $X$ and $Y$ small enough and it is finite when $G$ is a nilpotent Lie group. Moreover, in the nilpotent case we can endow $\fg$ with the product $(X, Y)\in \fg\times\fg\rightarrow X*Y=c(X, Y)$ such that $(\fg, *)$ is the simple connected, connected nilpotent Lie group with Lie algebra $\fg$.  

\section{Preliminaries}

This section is devoted to present the main background needed to establish the main theorem. We also prove some new results that will be useful ahead.

\subsection{Decompositions at the algebra level}

Let $\DC$ be a derivation of $\fg$ and $\alpha\in\C$ an eigenvalue of $\DC$. The real generalized eigenspaces of $\DC$ are given by  
$$
\mathfrak{g}_{\alpha}=\{X\in \mathfrak{g}:(\mathcal{D}-\alpha I)^{n}X=0\;\;\mbox{for some }n\geq 1\}, \;\;\mbox{ if }\;\;\alpha\in\R \;\;\mbox{ and}
$$
$$\mathfrak{g}_{\alpha}=\mathrm{span}\{\mathrm{Re}(v), \mathrm{Im}(v);\;\;v\in \bar{\fg}_{\alpha}\},\;\;\mbox{ if }\;\;\alpha\in\C$$
where $\bar{\fg}=\fg+i\fg$ is the complexification of $\fg$ and $\bar{\fg}_{\alpha}$ the generalized eigenspace of $\bar{\DC}=\DC+i\DC$, the extension of $\DC$ to $\bar{\fg}$.

Following \cite[Proposition 3.1]{SM1} it turns out that $[\bar{\fg}_{\alpha },\bar{\fg}_{\beta }]\subset 
\bar{\fg}_{\alpha +\beta }$ when $\alpha +\beta $ is an eigenvalue of $\mathcal{D}$ and zero otherwise. By considering in $\fg$ the subspaces $\fg_{\lambda}:=\bigoplus_{\alpha; \mathrm{Re}(\alpha)=\lambda}\fg_{\alpha}$, where $\fg_{\lambda}=\{0\}$ if $\lambda\in\R$ is not the real part of any eigenvalue of $\DC$, we get 
$$[\fg_{\lambda_1}, \fg_{\lambda_2}]\subset \fg_{\lambda_1+\lambda_2}\;\;\;\mbox{ when }\lambda_1+\lambda_2=\mathrm{Re}(\alpha)\;\mbox{ for some eigenvalue }\alpha\;\mbox{ of }\;\DC\;\mbox{ and zero otherwise}.$$

We define the {\it unstable, central }and {\it stable} subalgebras of $\fg$, respectively, by
\begin{equation*}
\mathfrak{g}^{+}=\bigoplus_{\alpha :\, \mathrm{Re}(\alpha)>	0}\mathfrak{g}_{\alpha },\hspace{1cm}\mathfrak{g}^{0}=\bigoplus_{\alpha :\,%
	\mathrm{Re}(\alpha )=0}\mathfrak{g}_{\alpha }\hspace{1cm}%
\mbox{ and }\hspace{1cm}\mathfrak{g}^{-}=\bigoplus_{\alpha :\, \mathrm{Re}%
	(\alpha )<0}\mathfrak{g}_{\alpha }.
\end{equation*}
It holds that $\mathfrak{g}^{+},\mathfrak{g}^{0}$ and $\mathfrak{g}^{-}$ are in fact $\DC$-invariant Lie subalgebras with $\mathfrak{g}^{+}$, $\mathfrak{g}^{-}$ nilpotent ones. Moreover, it $\fg$ decomposes as the direct sum $\mathfrak{g}=\mathfrak{g}^{+}\oplus \mathfrak{g}^{0}\oplus \mathfrak{g}^{-}$.

\subsection{Decompositions at the group level}

Let $G$ be a connected Lie group with Lie algebra $\mathfrak{g}$ identified with the set of left-invariant vector fields on $G$. A vector field $\mathcal{X}$ on $G$ is said to be \emph{linear} if for any $Y\in\fg$ it holds that 
$$[\XC, Y]\in \fg\;\;\;\mbox{ and }\;\;\;\XC(e)=0.$$
As showed in \cite[Theorem 1]{JPh1}, a linear vector field $\XC$ is complete and its associated flow $\{\varphi_t\}_{t\in \R}$ is a $1$-parameter subgroup of $\mathrm{Aut}(G)$ satisfying
\begin{equation*}
\forall g\in G, t\in\R\;\;\;\;\; \frac{d}{dt}\varphi_t(g)=\XC(\varphi_t(g))\;\;\;\mbox{ and }\;\;\;  (d\varphi_{t})_{e}=\mathrm{e}^{t\mathcal{D}}, \label{derivativeonorigin}
\end{equation*}
where $\DC:\fg\rightarrow\fg$ is the derivation defined by $\DC(Y):=[\XC, Y]$, and $\rme^{t\DC}$ its matrix exponential.

Let us denote by $G^{+}$, $G^{-}$, $G^{0}$, $G^{+,0},$ and $G^{-,0}$ the
connected Lie subgroups of $G$ with Lie algebras given by $\mathfrak{g}^{+}$, $%
\mathfrak{g}^{-}$, $\mathfrak{g}^{0}$, $\mathfrak{g}^{+,0}:=\mathfrak{g}%
^{+}\oplus \mathfrak{g}^{0}$ and $\mathfrak{g}^{-,0}:=\mathfrak{g}^{-}\oplus 
\mathfrak{g}^{0}$, respectively. By Proposition 2.9 of \cite{DS}, all the above subgroups are $\varphi$-invariant, closed and have trivial intersection, that is, 
$$G^+\cap G^-=G^+\cap G^{-, 0}=\ldots=\{e\}.$$
The subgroups $G^+, G^0$ and $G^-$ are called the unstable, central, and stable subgroups of $\XC$, respectively. We also, use the notation $G^{+, -}$ for the product of $G^+$ and $G^-$, that is, $G^{+, -}=G^+G^-$.

We say that $G$ is {\it decomposable} if 
$$G=G^{+, 0}G^-=G^{-, 0}G^+=G^{+, -}G^0.$$ 
By \cite[Proposition 3.3]{DSAyGZ} a sufficient condition for a group $G$ to be decomposable is the compactness of the central subgroup $G^0$. Moreover, on decomposable groups, any element can be uniquely decomposed into the product of factors in the stable, central, and unstable subgroups.

The following result explores some more properties coming from the assumption on the compactness of $G^0$.

\begin{lemma}
	\label{lemma1}
	Let $\fn$ to be the nilradical of $\fg$. If $G^0$ is a compact subgroup, then
	\begin{itemize}
		\item[1.] $G^+, G^-\subset N:=\exp(\fn)$;
		\item[2.] $N\cap G^0$ is a compact, connected and normal subgroup of $G$;
	\end{itemize}
\end{lemma}

\begin{proof}
	1. Let us denote by $R$ the solvable radical of $G$. Under the assumption that $G^0$ is compact, we get that $\left(G/R\right)^0=\pi(G^0)$ is also compact, where $\pi:G\rightarrow G/R$ is the canonical projection. Since $G/R$ is semi-simple, Proposition 3.3 of \cite{DSAyGZ} implies that $G/R=\left(G/R\right)^0$ and consequently that $G=G^0R$. In particular, $G^+, G^-\subset R$. On the other hand, Lemma 2.1 of \cite{DSAy1} assures that the nilradical $\fn$ of $\fr$ contains $\fg_{\alpha}$ for any nonzero eigenvalue $\alpha$ of $\DC$. Therefore, $\fg^+, \fg^-\subset\fn$ and hence $G^+, G^-\subset N$.
	
	2. To prove the second claim, let us notice that $N^0\subset N\cap G^0$ implying in particular that $N^0$ is a compact subgroup, and therefore $N$ is decomposable. By the uniqueness of the decomposition of each element in $N$, we must have that $G^0\cap N\subset N^0$ and hence $N^0=N\cap G^0$, showing that $N\cap G^0$ is a compact and connected Lie subgroup of $N$. 
	
	On the other hand, it is a standard fact that compact subgroups of nilpotent Lie groups are always central and hence $N^0\subset Z(N)$ (see, for instance \cite[Theorem 1.6]{ALEB}). Therefore, the fact that $G$ is decomposable implies by the previous item that $N\cap G^0$ is a normal subgroup of $G$ if and only if it is normalized by $G^0$. However, the fact that the nilradical is invariant by automorphisms, implies that $C_g(N)=N$ for any $g\in G$. In particular, if $g\in G^0$ and $h\in N\cap G^0$ we get that 
	$$G^0\ni ghg^{-1}=C_g(h)\in N\implies ghg^{-1}\in N\cap G^0,$$
	concluding the proof. 
\end{proof}

\bigskip

The next lemma shows that, in the decomposable case, if $G^{+, -}$ is a subgroup, then $G$ can be seen as a semi-direct product. 

\begin{lemma}
	\label{conj}
	If $G$ is decomposable and $G^{+, -}$ is a subgroup then $G$ is isomorphic to the semi-direct product $G^0\times_{\Ad}\fg^{+, -}$.
\end{lemma}

\begin{proof}
	Define the map
	$$\psi: G^0\times_{\Ad}\fg^{+, -}\rightarrow G,\;\;\;\;(g, x)\in G^0\times\fg^{+, -}\mapsto \exp(X) g\in G.$$
	If $p:G\times G\rightarrow G$ stands for the product in $G$ we have that 
	$$\psi=p\circ\left(\exp\times \id_G\right)|_{\fg^{+, -}\times G^0},$$
	and therefore $\psi$ is a continuous map. Furthermore, since $G$ is decomposable and $G^{+, -}$ is a nilpotent group, it holds that 
	$$G=G^{+, -}G^0=\exp(\fg^{+, -})G^0$$
	and so $\psi$ is surjective. Moreover, if $(g_1, X_1), (g_2, X_2)\in G^0\times\fg^{+, -}$, then
	$$\psi\left((g_1, X_1)(g_2, X_2)\right)=\psi(g_1g_2, X_1*\Ad(g_1)X_2)$$
	$$=\exp(X_1*\Ad(g_1)X_2)g_1g_2=\exp(X_1)\exp(\Ad(g_1)X_2)g_1g_2=\exp(X_1)g_1\exp(X_2)g_2=\psi(g_1, X_1)\psi(g_2, X_2),$$
	showing that $\psi$ is in fact a homomorphism.
	
	On the other hand, since $G^{+, -}\cap G^0=\{e\}$ it follows that
	$$(g, X)\in\ker\psi \;\;\iff \;\;\exp(X)g=e\iff G^{+, -}\ni\exp(X)=g^{-1}\in G^0\;\;\iff \;\;\exp(X)=g=e,$$
	and consequently 
	$$\psi\;\;\mbox{ is injective }\;\; \iff\;\; \ker\psi=\{(e, 0)\}\;\;\iff\;\; \exp:\fg^{+, -}\rightarrow G^{+, -} \;\;\mbox{ is injective.}$$
	However, since $\fg^{+, -}$ is a nilpotent Lie algebra, it holds that the exponential map $\exp:\fg^{+, -}\rightarrow G^{+, -}$ is a covering map, implying that $\exp^{-1}(e)\subset\fg^{+, -}$ is a discrete subset. On the other hand, $$\rme^{t\DC}(\exp^{-1}(e))\subset \exp^{-1}(e), \;\;\;\mbox{ for all }\;\;\;t\in\R,$$
	thus $\exp^{-1}(e)\subset\ker\DC\cap\fg^{+, -}=\{0\}$, which shows that $\exp$ is injective and consequently that $\psi$ is an isomorphism.
\end{proof}

\subsection{Linear systems}

Let $G$ be a connected Lie group with Lie algebra $\fg$ identified with the set of left-invariant vector fields on $G$ and $\Omega\subset\R^m$ a compact and convex subset containing the origin in its interior. The {\it set of the control functions} is by definition
$$\UC:=\{u:\R\rightarrow\R^m; \;u\;\mbox{ is a piecewise constant function with }  u(\R)\subset\Omega\}.$$
A \emph{linear system} on $G$ is given by the family of ordinary differential equations 
\begin{flalign*}
&&\dot{g}(t)=\XC(g(t))+\sum_{j=1}^mu_j(t)Y^j(g(t)),  &&\hspace{-1cm}\left(\Sigma_G\right)
\end{flalign*}
where the \emph{drift} $\XC$ is a linear vector field, $Y^1, \ldots, Y^m\in\fg$ and $u=(u_1, \ldots, u_m)\in\UC$. For any $g\in G$ and $u\in\UC$, the solution $t\mapsto\phi(t, g, u)$ of $\Sigma_G$ is complete and satisfies  
\begin{equation}
\label{prop}
\phi_{\tau, u}\circ L_g=L_{\varphi_{\tau}(g)}\circ\phi_{\tau, u}, \;\;\mbox{ for any }\;\;\tau\in\R, g\in G.
\end{equation}

The unstable, central and stable subgroups of $\Sigma_G$ are the ones induced by the linear drift of the system. As showed in \cite{DSAy1, DS} these subgroups strongly influences the dynamical behavior of the system.

\begin{definition}
	Let $F\subset G$ be a nonempty subset. We say that $g\in G$ is a {\it $F$-periodic point} of $\Sigma_G$ if there exist $f_1, f_2\in F$, $\tau_1, \tau_2>0$ and $u_1, u_2\in\UC$ such that 
	\begin{equation}
	\label{per}
	\phi(\tau_1, f_1, u_1)=g\;\;\;\;\mbox{ and }\;\;\;\;\phi(\tau_2, g, u_2)=f_2.
	\end{equation}
	A central periodic point is, by definition, a $G^0$-periodic point. We denote by $\mathrm{Per}(F; \Sigma_G)$ and by $\mathrm{Per}(\Sigma_G)$ the subsets of the $F$-periodic points and the central periodic points of $\Sigma_G$, respectively.
\end{definition}

\begin{remark}
	If $F$ is $\varphi$-invariant, then $F\subset \mathrm{Per}(F; \Sigma_G)$. In fact, if $g\in F$ we have by the $\varphi$-invariance of $F$ that
	$$f_1=\varphi_{-\tau}(g)\in F\;\;\mbox{ and }\;\;f_2=\varphi_{\tau}(g)\in F.$$
	By considering $u_1=u_2\equiv 0$ and $\tau_1=\tau_2=\tau$, we get
	$$\phi(\tau_1, f_1, u_1)=\varphi_{\tau}(f_1)=g\;\;\;\mbox{ and }\;\;\;\phi(\tau_2, g, u)=\varphi_{\tau}(g)=f_2,$$
	showing that $g$ is $F$-periodic.
\end{remark}

Next, we show that the whole curve connecting a point $g\in \mathrm{Per}(F, \Sigma_G)$ to $F$ is contained in $\mathrm{Per}(F, \Sigma_G)$.

\begin{lemma}
	\label{equi}
	For any $g\in G$, it holds that $g\in\mathrm{Per}(F; \Sigma_G)$ if and only if there exist $f\in F$, $\tau>0$ and $u\in \UC$ such that
	\begin{equation}
	\label{per2}
	g\in\{\phi(t, f, u), \;t\in(0, \tau)\}\;\;\;\mbox{ and }\;\;\;\phi(\tau, f, u)\in F.
	\end{equation}
	Hence, if $\phi(\tau, f, u)\in F$ for some $f\in F$, $\tau>0$ and $u\in\UC$, then $\phi(t, f, u)\in\mathrm{Per}(F; \Sigma_G)$ for any $t\in(0, \tau)$.
\end{lemma}

\begin{proof}
	Let $g\in \mathrm{Per}(F; \Sigma_G)$ and consider $f_1, f_2\in F$, $\tau_1, \tau_2>0$ and $u_1, u_2\in\UC$ satisfying (\ref{per}). The function
	$$
	u(t):=\left\{\begin{array}{cc}
	u_1(t), &  t\in (-\infty, \tau_1)\\
	u_2(t-\tau_1) &  t\in [\tau_1, +\infty)
	\end{array}\right.\;\;\mbox{ belongs to }\;\;\UC,$$
	and it holds that 
	$$\phi(\tau_1, f_1, u)=\phi(\tau_1, f_1, u_1)=g,\;\;\;\mbox{ and }\;\;\;\phi(\tau_1+\tau_2, f_1, u)=\phi(\tau_2, \phi(\tau_1, f_1, u_1), \theta_{\tau_1}u)=\phi(\tau_2, g, u_2)=f_2\in F,$$
	showing that relation (\ref{per2}) holds. Reciprocally, let $f\in F$, $\tau>0$ and $u\in\UC$ such that relation (\ref{per2}) is satisfied. Let $\tau_1\in(0, \tau)$ and $f_2\in F$ such that 
	$$g=\phi(\tau_1, f, u)\;\;\;\;\mbox{ and }\;\;\;\;f_2=\phi(\tau, f, u)$$
	and set $f_1=f$, $\tau_2=\tau-\tau_1>0$, $u_1=u$ and $u_2=\theta_{\tau_1}u$. Then 
	$$\phi(\tau_1, f_1, u_1)=\phi(\tau_1, f, u)=g,$$
	and
	$$\phi(\tau_2, g, u_2)=\phi(\tau_2, \phi(\tau_1, f, u), \theta_{\tau_1}u)=\phi(\tau_2+\tau_1, f, u)=\phi(\tau, f, u)=f_2,$$
	showing that $g$ is $F$-periodic, and concluding the proof.
\end{proof}

\begin{remark}
In the particular case where $F=\{g\}$, the previous lemma shows that the set $\mathrm{Per}(g; \Sigma_G)$ consists of closed orbits passing by $g\in G$. 
\end{remark}

Let $G$ and $H$ connected Lie groups. We say that two linear systems $\Sigma_G$ and $\Sigma_H$, respectively on $G$ and $H$, are {\it conjugated} if there exists a surjective homomorphism
 $\psi:G\rightarrow H$ such that
\begin{equation}
\label{solu}
\psi\left(\phi^G(t, g, u)\right)=\phi^H(t, \psi(g), u)), \;\;\;\mbox{ for any }\;\;g\in G, \,t\in\R\;\mbox{ and }u\in\UC.
\end{equation}
The map $\psi$ is said to be a {\it conjugation} between $\Sigma_G$ and $\Sigma_H$. 

\begin{proposition}
	\label{control}
Let $\psi$ be a conjugation between $\Sigma_G$ and $\Sigma_H$. If $F\subset G$ is a nonempty subset satisfying $\ker\psi\cdot F=F$, then  
 $$\psi\left(\mathrm{Per}(F; \Sigma_G)\right)=\mathrm{Per}(\psi(F); \Sigma_{H})\;\;\;\;\mbox{ and }\;\;\;\;\psi^{-1}\left(\mathrm{Per}(\psi(F); \Sigma_{H})\right)=\mathrm{Per}(F; \Sigma_G).$$ 
 In particular, if $\ker\psi$ is a compact subgroup, then $\mathrm{Per}(F; \Sigma_G)$ is a bounded subset of $G$ if and only if $\mathrm{Per}(\psi(F); \Sigma_{H})$ is a bounded subset of $H$.
\end{proposition}

\begin{proof} To prove the first statement, we remark that due to the relationship between the involved sets it is enough to prove that 
	$$\psi\left(\mathrm{Per}(F; \Sigma_G)\right)\subset\mathrm{Per}(\psi(F); \Sigma_{H})\;\;\;\;\mbox{ and }\;\;\;\;\psi^{-1}\left(\mathrm{Per}(\psi(F); \Sigma_{H})\right)\subset\mathrm{Per}(F; \Sigma_G).$$ 
	
	If $g\in \mathrm{Per}(F; \Sigma_G)$ it follows from (\ref{per2}) that
	$$g\in\{\phi^G(t, f, u), \;t\in(0, \tau)\}\;\;\;\mbox{ and }\;\;\;\phi^G(\tau, f, u)\in F,$$
	for some $f\in F$, $\tau>0$ and $u\in\UC$.
	By equation (\ref{solu}) we get 
	$$\psi(g)\in\psi\left(\{\phi^G(t, f, u), \;t\in(0, \tau)\}\right)=\{\psi\left(\phi^G(t, f, u)\right), \;t\in(0, \tau)\}=\{\phi^H(t, \psi(f), u), \;t\in(0, \tau)\}$$
and 
	$$\psi(\phi^G(\tau, f, u))\in\psi(F)\iff \phi^H(\tau, \psi(f), u)\in \psi(F),$$
	showing that 
	$$\psi\left(\mathrm{Per}(F; \Sigma_G)\right)\subset\mathrm{Per}(\psi(F); \Sigma_{H}).$$ 
	
	Reciprocally, if $g\in\psi^{-1}\left(\mathrm{Per}(\psi(F); \Sigma_{H})\right)$ there exist $f_1, f_2\in F$, $\tau_1, \tau_2>0$ and $u_1, u_2\in\UC$ satisfying
	$$\phi^H(\tau_1, \psi(f_1), u_1)=\psi(g)\;\;\;\;\mbox{ and }\;\;\;\;\phi^H(\tau_2, \psi(g), u_2)=\psi(f_2),$$
	and hence,
	$$h_1\phi^G(\tau_1, f_1, u_1)=g\;\;\;\;\mbox{ and }\;\;\;\;\phi^G(\tau_2, g, u_2)=h_2f_2, \;\;\mbox{ for some }\;\;h_1, h_2\in \ker\psi.$$
	Using equation (\ref{solu}) for the control $u\equiv 0$ gives us that 
	$$\psi\circ\varphi_t^G=\varphi^H\circ\psi, \;\;\;\forall\; t\in\R\;\;\;\;\mbox{ and hence }\;\;\;\;\varphi_t^G(\ker\psi)\subset\ker\psi, \;\;\;\forall\; t\in\R.$$
	By defining $\tilde{f}_1:=\varphi^G_{-\tau_1}(h_1)f_1$ and $\tilde{f}_2:=h_2f_2$ the assumption that $\ker\psi\cdot F=F$ implies that $\tilde{f}_1, \tilde{f}_2\in F$. Moreover, by equation (\ref{prop})
	$$\phi^G(\tau_1, \tilde{f}_1, u_1)=\phi^G(\tau_1, f_1\varphi^G_{-\tau_1}(h_1), u_1)=h_1\phi^G(\tau_1, f_1, u_1)=g$$
	and
	$$\phi^G(\tau_2, g, u_2)=h_2f_2=\tilde{f}_2,$$
	showing that
	$$\psi^{-1}\left(\mathrm{Per}(\psi(F); \Sigma_{H})\right)\subset\mathrm{Per}(F; \Sigma_G),$$ 
	as stated.
	
	For the last assertion, it certainly holds that $\mathrm{Per}(\psi(F); \Sigma_{H})$ is bounded as soon as $\mathrm{Per}(F; \Sigma_G)$ is bounded. Reciprocally, if $\mathrm{Per}(\psi(F); \Sigma_{H})$ is bounded, there exists a compact set $K\subset G$ such that $\mathrm{Per}(\psi(F); \Sigma_{H})\subset\psi(K)$ which, by the previous equalities, implies that
	$$\mathrm{Per}(F; \Sigma_G)=\psi^{-1}\left(\mathrm{Per}(\psi(F); \Sigma_{H})\right)\subset\psi^{-1}(\psi(K))=K\ker\psi.$$
	Consequently $\mathrm{Per}(F; \Sigma_G)$ is a bounded subset if $\ker\psi$ is a compact subgroup of $G$.
\end{proof}

\section{A particularly important case}

In this section we analyze the solutions of a linear system on a semi-direct product of Lie groups. This analysis will be crucial in reducing the hypothesis of our main result.

Let $\fu$ be a nilpotent Lie algebra and identify it with the connected, simply connected Lie group $(\fu, *)$, where the product is given by the BCH formula. This identification between algebra and group allows us to work indistinctly with their elements. We will however use capital letters $X, Y, Z, \ldots$ for the elements in $\fu$ seen as Lie algebra and small letters $x, y, z, \ldots$ for the members in $\fu$ seen as Lie group. By the previous identifications, it turns out that a linear vector field $\XC$ on $\fu$ coincides with its associated derivation. In fact, since in this case $\{\varphi_t\}_{t\in\R}$ is one-parameter subgroup of automorphisms in $\mathrm{Aut}(\fu)$, there exists a derivation $\DC\in\mathrm{Lie}(\mathrm{Aut}(\fu))=\mathrm{Der}(\fu)$ such that $\varphi_t=\rme^{t\DC}$ for any $t\in\R$. Consequently, 
$$\XC(x)=\frac{d}{ds}_{|s=0}\varphi_s(x)=\frac{d}{ds}_{|s=0}\rme^{s\DC}x=\DC x$$

\bigskip

Let $H$ be a connected Lie group with Lie algebra $\fh$, $\rho:H\rightarrow\mathrm{Aut}(\fu)$ a continuous homomorphism and consider the semi-direct product $H\times_{\rho}\fu$. Since $T_{(e, 0)}(H\times_{\rho}\fu)=\fh\times\fu$, for any given $(Y, Z)\in\fh\times\fu$ the associated left-invariant vector field $(Y, Z)$ on $H\times_{\rho}\fu$ is, by definition, $(Y, Z)(h, x)=(dL_{(h, x)})_{(e, 0)}\alpha'(0)$, where $\alpha:(-\varepsilon, \varepsilon)\rightarrow H\times_{\rho}\fu$ is any differentiable curve satisfying $\alpha(0)=(e, 0)$ and $\alpha'(0)=(Y, Z)$. The curve  $\alpha(s):=(\exp(sY), sZ)$ satisfies the previous conditions and 
$$L_{(h, x)}(\alpha(s)))=(h, x)(\exp(sY), sZ)=\left(h\exp(sY), x*s\rho(h)Z\right).$$
Through the associated BCH formula, we obtain  
$$x*s\rho(h)Z=x+s\rho(h)Z+\frac{1}{2}[x, s\rho(h)Z]+\frac{1}{12}\left([x, [x, s\rho(h)Z]]+[s\rho(h)Z, [s\rho(h)Z, x]]\right)+\cdots,$$
and by differentiating at $s=0$, we get
$$
(Y, Z)(h, x)=\left(Y(h), (\rho(h)Z)(x)\right), \;\mbox{ where }\; (\rho(h)Z)(x):=\sum_{p=0}^{k-1}c_p\ad(x)^p\rho(h)Z.
$$
Here $k\in\N$ is the smallest natural number such that $\ad(x)^k\equiv 0$ for all $x\in \fu$, which exists by the fact that $\fu$ is nilpotent, and the coeficients $c_p$ are the ones given by the BCH formula. For instance, $c_0=1$, $c_1=-1/2$, $c_2=1/12$, and so on. 

If $\XC$ is a linear vector field on $G$ and $\DC$ is a derivation of $\fu$, the map
$$(h, x)\in H\times_{\rho} \fu\mapsto (\XC(h), \DC x)\in T_{(h, x)}(H\times_{\rho}),$$
is a linear vector field on $H\times_{\rho}\fu$.
Consider the linear system 
\begin{equation}
\label{semi}
\left\{
\begin{array}{l}
\dot{h}=\XC(h)+\sum_{j=1}^mu_jY_j(h)\\
\dot{x}=\DC x+\sum_{j=1}^mu_j(\rho(h)Z_j)(x).
\end{array}
\right.
\end{equation}

Let us fix $(h, x)\in H\times_{\rho}\fu$ and $u\in\UC$. In order to avoid cumbersome notation, we denote by $t\mapsto (h_t, x_t)$ the solution of \ref{semi} associated with $u\in\UC$ and initial condition $(h_0, x_0)=(h, x)$. Consider the continuous curve on $\fu$ given by
$$t\in\R\mapsto Z_t\in\fu, \;\;\;\mbox{ where }\;\;Z_t:=\rho(h_t)\left(\sum_{j=1}^mu_j(t)Z_j\right).$$
Since $\rho(h_t)\in\mathrm{Aut}(\fu)$ is a linear map, 
$$\sum_{j=1}^mu_j(t)(\rho(h_t)Z_j)(x_t)=\sum_{j=1}^mu_j(t)\sum_{p=0}^nc_p\ad(x_t)^p\rho(h_t)Z_j=\sum_{p=0}^nc_p\ad(x_t)^pZ_t=Z_t(x_t).$$
In particular, the second equation in (\ref{semi}) can be rewritten as 
\begin{equation}
\label{alternative}
\dot{x}_t=\DC x_t+Z_t(x_t).
\end{equation}

Consider $\fu^1\supset \fu^2\supset\cdots\supset\fu^k\supset\fu^{k+1}=\{0\}$ to be the central series of $\fu$ defined as
$$\fu^1=\fu \;\;\;\;\mbox{ and }\;\;\;\;\fu^{i+1}=[\fu^i, \fu], \;\;\;\;\mbox{ for }\;\;\;i\in\{1, \ldots, k\}.$$
Let $i\in\{1, \ldots, k\}$ and choose $V_i\subset \fu^i$ to be a complementary space of $\fu^{i+1}$ in $\fu^i$, that is, 
\begin{equation}
\label{deco}
V_i\oplus\fu^{i+1}=\fu^i.\hspace{1cm}\mbox{ In particular, }\;\fu^i=\bigoplus_{l=i}^kV_l, \;\;\mbox{ for }\;\;i\in\{1, \ldots, k\}.
\end{equation}
We use the notation $x=(x^1, \ldots, x^k)$ to emphasize the decomposition of $x\in \fu$ in the components $x^i\in V_i$. We aim to write the solutions of the system (\ref{alternative}) using the decomposition (\ref{deco}) in the $V_i$-components.  

For any derivation $\DC$, the fact that $\DC\fu^i\subset \fu^i$ gives us a block-triangular decomposition form
\begin{equation}
\label{block}
\DC=\left(\begin{array}{ccccc}
\DC_{11} &     0    &     0    & \cdots &    0    \\
\DC_{21} & \DC_{22} &     0    & \cdots &    0    \\
\DC_{31} & \DC_{32} & \DC_{33} & \cdots &    0    \\
\vdots   & \vdots   & \vdots   & \ddots &  \vdots \\
\DC_{k1} & \DC_{k2} & \DC_{k3} & \cdots & \DC_{kk}
\end{array}\right), \;\;\mbox{ where }\;\;\DC_{ij}:V_j\rightarrow V_i\;\;\mbox{ is a linear map}.
\end{equation}
Since for any $x\in \fu$ the map $\ad(x)$ is a derivation, we can consider its block-triangular decomposition as previously. Let $p\in\{1, \ldots, k-1\}$ and consider the linear map $B^p_{i, j}(x): V_j\rightarrow V_i$, obtained from the block-triangular decomposition of $\ad(x)^p$. Our subsequent analysis will be based on the following lemma.

\begin{lemma}
	With the previous notations, for any $p\in\{1, \ldots, k-1\}$ it holds  
	$$B^p_{ij}(x)=\left\{\begin{array}{cc} 0 & \;\mbox{ for }\;i<p+j\\B_{ij}^p\left(x^1, \ldots, x^{i-j-p+1}\right) & \;\mbox{ for }\;i\geq p+j \end{array}\right.,$$
	when $x=(x^1, \ldots, x^k)$.
\end{lemma}
\begin{proof}
	The proof will proceed by induction. Since for any $V_i\subset \fu^i$ we have that $\ad(x_l) V_j\subset\fu^{j+l}=\bigoplus_{q={j+l}}^kV_q$ for any $x_l\in V_l$. Hence,  $B_{ij}(x^l)=0$ for any $x_l\in V_l$ if $i<l+j$, implying that
$$\ad(x)=\sum_{l=1}^{k-1}\ad(x^l)\;\;\; \implies \;\;\;B_{ij}(x)=\left\{\begin{array}{cc} 0 & \;\mbox{ for }\;i<j+1\\ B_{ij}\left(x^1, \cdots, x^{i-j}\right) & \;\mbox{ for }\;i\geq j+1 \end{array}\right.,$$
and showing the result for $p=1$. If the result is true for $p$, a simple calculation shows that 
$$B^{p+1}_{ij}(x)=\sum_{l=1}^kB_{il}(x)B^p_{lj}(x).$$
By inductive hypothesis, it holds  
$$B_{il}(x)=\left\{\begin{array}{cc} 0 & \;\mbox{ for }\;i<1+l\\B_{il}\left(x^1, \ldots, x^{i-l}\right) & \;\mbox{ for }\;i\geq 1+l \end{array}\right.\mbox{ and }\;\;B^p_{lj}(x)=\left\{\begin{array}{cc} 0 & \;\mbox{ for }\;l<p+j\\B_{lj}^p\left(x^1, \ldots, x^{l-j-p+1}\right) & \;\mbox{ for }\;l\geq p+j \end{array}\right..$$
Therefore, $B^{p+1}_{ij}(x)=0$ for $i<(p+1)+j$ and
$$B^{p+1}_{ij}(x)=\sum_{p+j\leq l\leq i-1}B_{il}\left(x^1, \ldots, x^{i-l}\right)B^p_{lj}\left(x^1, \ldots, x^{l-j-p+1}\right),$$
which certainly only depends on $x^1, \ldots, x^{i-j-(p+1)+1}$ and thus 
$$B^{p+1}_{ij}(x)=B^{p+1}_{ij}\left(x^1, \ldots, x^{i-j-(p+1)+1}\right)\;\; \mbox{ if }\;\;i\geq(p+1)+j,$$
concluding the proof.
\end{proof}

\bigskip

On the other hand, if $Z=(Z^1, \ldots, Z^k)\in\fu$, then 
$$\left(\ad(x)^pZ\right)^i=\sum_{j=1}^kB^p_{ij}(x)Z^j=\sum_{j=1}^{i-p}B^p_{ij}(x^1, \ldots, x^{i-j-p+1})Z^j,\;\;\mbox{ if }\;\;p<i$$
and $\left(\ad(x)^pZ\right)^i=0$ if $p\geq i$.  Therefore, 
$$\left(Z(x)\right)^i=\left(\sum_{p=0}^{k-1}c_p\ad(x)^pZ\right)^i=\sum_{p=0}^{k-1}c_p\left(\ad(x)^pZ\right)^i=Z^i+\sum_{p=1}^{i-1}c_p\sum_{j=1}^{i-p}B^p_{ij}(x^{1}, \ldots, x^{i-j-p+1})Z^j,$$
and the $V_i$-component of $Z(x)$ just depends on $x^1, \ldots, x^{i-1}$. For any $i\in\{1, \ldots, k\}$ we can define the continuous map
$G^i:V_1\times\cdots\times V_{i-1}\times \fu\rightarrow V_i$ by
$$G^1(Z):=Z^1\;\;\;\mbox{ and }\;\;\;G^i(x^1, \ldots, x^{i-1}; Z):=\sum_{j=1}^{i-1}\DC_{ij}x^{j}+\left(Z(x)\right)^i, \;\;\;\mbox{ for }\;i\geq 2.$$

In the sequel, we give a decomposition of (\ref{alternative}) in terms of the maps $G^i$ above.

\begin{theorem}
	\label{sol}
	With the previous notations, the system (\ref{alternative}) reads in coordinates as
	\begin{equation}
	\label{components}
	\left\{\begin{array}{l}
	\dot{x}^1_t=\DC_{11}x^1_t+G^1\left(Z_t\right)\\
	\dot{x}^2_t=\DC_{22}x^2_t+G^2\left(x^1_t; Z_t\right)\\
	\dot{x}^3_t=\DC_{33}x^3_t+G^3\left(x^1_t; x^2_t; Z_t\right)\\
	\hspace{2,3cm}\vdots\\
	\dot{x}^k_t=\DC_{kk}x^k_t+G^k\left(x^1_t, \ldots, x^{k-1}_t;Z_t\right)
	\end{array}\right..
	\end{equation}
	Moreover, for any $i\in\{1, \ldots, k\}$ it holds that  
	\begin{equation}
	\label{solutiondirect}
	x_t^i=\int_0^t\rme^{(t-s)\DC_{ii}}G^i\left(x^1_s, \ldots, x^{i-1}_s; Z_s\right)ds+\rme^{t\DC_{ii}}x_0^i, \;\;\mbox{ for }\;\;i=1, \ldots, k.
	\end{equation}
\end{theorem}

\begin{proof}
	Since 
	$$(\DC x)^i=\sum_{j=1}^i\DC_{ij}x^j=\DC_{ii}x^i+\sum_{j=1}^{i-1}\DC_{ij}x^j,$$
	we get
	$$\dot{x}^i_t=\left(\DC x_t+Z_t(x_t)\right)^i=\left(\DC x_t\right)^i+\left(Z_t(x_t)\right)^i=\DC_{ii}x_t^i+\sum_{j=1}^{i-1}\DC_{ij}x_t^j+\left(Z_t(x_t)\right)^i=\DC_{ii}x^i_t+G^i\left(x^1_t, \ldots, x^{i-1}_t; Z_t\right),$$
	which proves equations (\ref{components}). Equation (\ref{solutiondirect}) follows direct from integration. In fact,
	$$\dot{x}^i_t=\DC_{ii}x^1_t+G\left(x^1_t, \ldots, x_t^{i-1}; Z_t\right)\iff \frac{d}{dt}\rme^{-t\DC_{ii}}x^i_t=\rme^{-t\DC_{ii}}G\left(x^1_t, \ldots, x_t^{i-1}; Z_t\right)$$
	$$\iff \rme^{-t\DC_{ii}}x^i_t-x_0^i=\int_0^t\rme^{-s\DC_{ii}}G\left(x^1_s, \ldots, x_s^{i-1}; Z_s\right)ds\iff x^i_t=\int_0^t\rme^{(t-s)\DC_{ii}}G\left(x^1_s, \ldots, x_s^{i-1}; Z_s\right)ds+\rme^{t\DC_{ii}}x_0^i.$$ 		 
\end{proof}

\begin{remark}
	It is relevant to notice that in the previous calculations we have also the dependence on the point $(h, x)\in H\times_{\rho}\fu$ and on the control $u\in\UC$ which we previously fixed. If we want to emphasize the dependence on these parameters we will use the notations 
	$$h_{t, u}, \;\;\;x_{t, u, h}, \;\;\;\mbox{ and }\;\;\;Z_{t, u, h}.$$
\end{remark}

\bigskip

The next lemma will be central in the proof of our main result.

\begin{lemma}
	\label{central}
	Assume that $H$ is a compact group and that $\DC$ has only eigenvalues with nonzero real part. Then, the set of central periodic points of the linear system (\ref{semi}) is bounded.	
\end{lemma}

\begin{proof} Under the lemma's assumptions $(H\times_{\rho}\fu)^0=H\times\{0\}$. For simplicity, consider $\mathbf{P}:=\mathrm{Per}\left(\Sigma_{H\times_{\rho}\fu}\right)$. To prove the lemma we first remark that, since $H$ is a compact group, we only have to show that $\pi_2\left(\mathbf{P}\right)$ is bounded in $\fu$, where $\pi_2$ is the projection onto the second factor. By defining 
	$$\pi_{2, i}:H\times_{\rho}\fu\rightarrow V_i, \;\;\; (h, (x^1, \ldots, x^k))\mapsto x^i,$$
it follows that $\pi_2\left(\mathbf{P}\right)$ is bounded in $\fu$ if and only if $\pi_{2, i}\left(\mathbf{P}\right)$ is bounded in $V_i$ for any $i=1, \ldots, k$, which we will prove recurrently after some preliminaries.

By the block-triangular decomposition form of $\DC$ given in (\ref{block}), it follows that if $\DC$ has only eigenvalues with nonzero real part the same is true for $\DC_{ii}$ for any $i\in\{1, \ldots, k\}$. Under such assumption, for any $i\in\{1, \ldots, k\}$ we can consider the decomposition $V_i=V_i^+\oplus	V_i^-$, where $V_i^{+} (\mbox{resp. }V_i^-)$ is the sum of the real generalized eigenspaces of $\DC_{ii}$ associated with eigenvalues with positive (resp. negative) real parts. Therefore, if $|\cdot|$ is a norm in $\fu$ there exist constants $\kappa_i, \mu_i>0$ such that 
$$|\rme^{t\DC_{ii}}\pi_i^-(x^i)|\leq \kappa_i\rme^{-t\mu_i}|\pi_i^-(x^i)|\;\;\;\;\mbox{ and }\;\;\;\;|\rme^{-t\DC_{ii}}\pi_i^+(x^i)|\leq \kappa_i\rme^{-t\mu_i}|\pi^+(x^i)|, \;\;\mbox{ for any }\; t> 0, \;x^i\in V_i,$$
where $\pi_i^{\pm}:V_i\rightarrow V_i^{\pm}$ are the projections associated with the decomposition $V_i=V_i^+\oplus V_i^-$. Let us fix $M_1>0$ such that 
$$|G^1\left(Z_{t, u, h}\right)|\leq M_1, \;\;\;\mbox{ for all }\;\;\;t\geq 0,\, h\in H\;\mbox{ and }\; u\in\UC,$$
which exists by the compactness of $H\times\UC$ and the continuity of $G^1$.

Let then $x^1\in\pi_{2, 1}\left(\mathbf{P}\right)$ and consider $h_1, h_2\in H$, $\tau_1, \tau_2>0$ and $u_1, u_2\in\UC$ such that 
$$x^1_{0, u_1, h_1}=x^1_{\tau_2, u_2, h_2}=0\;\;\;\;\mbox{ and }\;\;\;\;x^1_{0, u_2, h_2}=x^1_{\tau_1, u_1, h_1}=x^1.$$
Through Theorem \ref{sol} we obtain 
$$x^1=\int_0^{\tau_1}\rme^{(\tau_1-s)\DC_{11}}G^1\left(Z_{s, u_1, h_1}\right)ds\;\;\;\;\;\mbox{ and }\;\;\;\;\;0=\int_0^{\tau_2}\rme^{(\tau_2-s)\DC_{11}}G^1\left(Z_{s, u_2, h_2}\right)ds+\rme^{\tau_2\DC_{11}}x^1,$$
implying that
$$\hspace{-6,5cm}\left|\pi_1^-(x^{1})\right|\leq \int_{0}^{\tau_1}\left|\rme^{(\tau_1-s)\DC_{11}}\pi^-_1\left(G^1\left(Z_{s, u_1, h_1}\right)\right)\right|ds$$
$$\leq \int_{0}^{\tau_1}\kappa_1\rme^{-(\tau_1-s)\mu_1}\left|\pi^-\left(G^1(Z_{s, u_1, h_1})\right)\right|ds\leq \frac{\kappa_1}{\mu_1}M_1(1-\rme^{-\tau_1\mu_1})\leq \frac{\kappa_1}{\mu_1}M_1,$$
and also
$$\hspace{-7cm}\left|\pi_1^+(x^{1})\right|\leq \int_{0}^{\tau_2}\left|\rme^{-s\DC_{11}}\pi^+_1\left(G^1(Z_{s, u_2, h_2})\right)\right|ds$$
$$\hspace{-0.7cm}\leq \int_{0}^{\tau_2}\kappa_1\rme^{-s\mu_1}\left|\pi^+\left(G^1(Z_{s, u_2, h_2})\right)\right|ds\leq \frac{\kappa_1}{\mu_1}M_1(1-\rme^{-\tau_2\mu_1})\leq \frac{\kappa_1}{\mu_1}M_1.$$
Consequently,
$$\left|x^1\right|=\left|\pi^+_1(x^1)+\pi^-_1(x^1)\right|\leq \left|\pi^+_1(x^1)\right|+\left|\pi^-_1(x^1)\right|\leq 2\frac{\kappa_1}{\mu_1}M_1,$$
proving the boundedness of $\pi_{2, 1}\left(\mathbf{P}\right)$.
 
Let $i\in\{2, \ldots, k\}$ and assume that $\pi_{2, j}\left(\mathbf{P}\right)\subset V_j$ is a bounded set for $j<i$. If $x^i\in \pi_{2, i}\left(\mathbf{P}\right)$, it holds that 
$$x^i_{0, u_1, h_1}=x^i_{\tau_2, u_2, h_2}=0\;\;\;\;\mbox{ and }\;\;\;\;x^i_{0, u_2, h_2}=x^i_{\tau_1, u_1, h_1}=x^i$$
for some $h_1, h_2\in H$, $\tau_1, \tau_2>0$ and $u_1, u_2\in\UC$. Again by Theorem \ref{sol} we get  
$$x^i=\int_0^{\tau_1}\rme^{(\tau_1-s)\DC_{ii}}G\left(x^1_{s, u_1, h_1}, \ldots, x^{i-1}_{s, u_1, h_1}; Z_{s, u_1, h_1}\right)ds,$$
and
$$0=\int_0^{\tau_2}\rme^{(\tau_2-s)\DC_{ii}}G\left(x^1_{s, u_2, h_2}, \ldots, x^{i-1}_{s, u_2, h_2}; Z_{s, u_2, h_2}\right)ds+\rme^{\tau_2\DC_{ii}}x^i.$$
On the other hand, by inductive hypothesis, $\pi_{2, j}\left(\mathbf{P}\right)$ is a bounded set for $j=1, \ldots, i-1$. Hence, the continuity of $G^i$ assures the existence of $M_i>0$ such that 
$$\left|G_i\left(x^1, \ldots, x^{i-1}; Z_{t, u}\right)\right|\leq M_i, \;\;\;\mbox{ for any }\;\;\; x^j\in \pi_{2, j}\left(\mathbf{P}\right), \;\;j=1, \ldots, i-1, \;t>0, \;u\in\UC.$$
By Lemma \ref{equi} it holds that, 
$$x^j_{s, u_1, h_1}\in\pi_{2, j}\left(\mathbf{P}\right), \;\;\;\;s\in[0, \tau_1]\;\;\;\;\mbox{ and }\;\;\;\; x^j_{s, u_2, h_2}\in\pi_{2, j}\left(\mathbf{P}\right), \;\;\;\;s\in[0, \tau_2],$$
for any $j=1, \ldots, i-1$. Hence,
$$\hspace{-6,4cm}\left|\pi_i^-(x^i)\right|\leq \int_{0}^{\tau_1}\left|\rme^{(\tau_1-s)\DC_{ii}}\pi^-_i\left(G\left(x^1_{s, u_1, h_1}, \ldots, x^{i-1}_{s, u_1, h_1}; Z_{s, u_1, h_1}\right)\right)\right|ds$$
$$\leq \int_{0}^{\tau_1}\kappa_i\rme^{-(\tau_1-s)\mu_i}\left|\pi^-_i\left(G\left(x^1_{s, u_1, h_1}, \ldots, x^{i-1}_{s, u_1, h_1}; Z_{s, u_1, h_1}\right)\right)\right|ds\leq \frac{\kappa_i}{\mu_i}M_i(1-\rme^{-\tau_1\mu_i})\leq \frac{\kappa_i}{\mu_i}M_i$$
and 
$$\hspace{-6,3cm}\left|\pi_i^+(x^{i})\right|\leq \int_{0}^{\tau_2}\left|\rme^{-s\DC_{ii}}\pi^+_i\left(G\left(x^1_{s, u_2, h_2}, \ldots, x^{i-1}_{s, u_2, h_2}; Z_{s, u_2, h_2}\right)\right)\right|ds$$
$$\leq \int_{0}^{\tau_2}\kappa_i\rme^{-s\mu_i}\left|\pi^+_i\left(G\left(x^1_{s, u_2, h_2}, \ldots, x^{i-1}_{s, u_2, h_2}; Z_{s, u_2, h_2}\right)\right)\right|ds\leq \frac{\kappa_i}{\mu_i}M_i(1-\rme^{-\tau_2\mu_i})\leq \frac{\kappa_i}{\mu_i}M_i,$$
implying that
$$\left|x^i\right|=\left|\pi^+_i(x^i)+\pi^-_i(x^i)\right|\leq \left|\pi^+_i(x^i)\right|+\left|\pi^-_i(x^i)\right|\leq 2\frac{\kappa_i}{\mu_i}M_i.$$
Since $x^i\in\pi_{2, i}\left(\mathbf{P}\right)$ is arbitrary, we get that $\pi_{2, i}\left(\mathbf{P}\right)$ is a bounded set, finishing the proof.	
\end{proof}

\section{The main result}

In this section we show our main result, namely, that the compactness of the central subgroup of a linear system is a necessary and sufficient condition for the boundedness of the central periodic points of $\Sigma_G$.

\begin{theorem}
	\label{teo}
	Let $G$ be connected Lie group and $\Sigma_G$ a linear system on $G$. Then, the central subgroup of $\Sigma_G$ is a compact subgroup if, and only if, the set of the central periodic points of $\Sigma_G$ is bounded.
\end{theorem}

\begin{proof} Assume that $\mathrm{Per}(\Sigma_G)$ is a bounded subset of $G$. Since $G^0$ is $\varphi$-invariant, it follows that $G^0\subset \mathrm{Per}(\Sigma_G)$. Therefore, the fact that $G^0$ is a closed subgroup of $G$ together with the previous inclusion implies the compactness of $G^0$.
	
Reciprocally, assume that $G^0$ is a compact subgroup and consider first the case where $G^{+, -}$ is a subgroup of $G$. Since the compactness of $G^0$ implies the decomposability of $G$, Lemma \ref{conj} shows that the map
$$\psi: \;\;(g, x)\in G^0\times_{\Ad}\fg^{+, -}\mapsto\exp(X) g\in G,$$
is an isomorphism. Therefore, the maps $\psi$ is a conjugation between an induced linear control system $\Sigma_{G^0\times_{\Ad}\fg^{+, -}}$ and $\Sigma_G$.

Moreover, for any $(g, X)\in G^0\times_{\Ad}\fg^{+, -}$ we obtain
$$\psi\left(\varphi_t|_{G^0}(g), \rme^{t\DC|_{\fg^{+, -}}}X\right)=\psi\left(\varphi_t(g), \rme^{t\DC} X\right)=\exp(\rme^{t\DC}X)\varphi_t(g)=\varphi_t(\exp(X))\varphi_t(g)=\varphi_t(\exp(X)g)=\varphi_t(\psi(g, X)),$$
and by differentiation, we get $\psi_*\circ(\XC_{|G^0}\times\DC|_{\fg^{+, -}})=\XC\circ\psi$. Thus, the linear vector field associated to $\Sigma_{G^0\times_{\Ad}\fg^{+, -}}$ is given by $\XC_{|G^0}\times\DC|_{\fg^{+, -}}$. In particular, the linear system $\Sigma_{G^0\times_{\Ad}\fg^{+, -}}$ is of the form (\ref{semi}). The group $G^0$ is compact and $\DC|_{\fg^{+, -}}$ has only eigenvalues with nonzero real part, thus Lemma \ref{central} implies that 
$$\mathrm{Per}\left(\Sigma_{G^0\times_{\Ad}\fg^{+, -}}\right)\;\;\;\mbox{ is a bounded subset}.$$
Since $\psi$ is an isomorphism, Proposition \ref{control} shows that
$$\psi\left(\mathrm{Per}\left(\Sigma_{G^0\times_{\Ad}\fg^{+, -}}\right)\right)=\psi\left(\mathrm{Per}\left(G^0\times\{0\}, \Sigma_{G^0\times_{\Ad}\fg^{+, -}}\right)\right)=\mathrm{Per}\left(\psi(G^0\times\{0\}), \Sigma_G\right)=\mathrm{Per}\left(G^0, \Sigma_G\right)=\mathrm{Per}\left(\Sigma_G\right)$$
and hence $\mathrm{Per}\left(\Sigma_G\right)$ is a bounded subset of $G$, when $G^{+, -}$ is a subgroup.

To prove the general statement, let us consider as previously the compact, connected normal subgroup of $G$ given by $N^0=N\cap G^0$. We know that $\widehat{G}=N^0\setminus G$ is a Lie group and the induced system $\Sigma_{\widehat{G}}$ is linear. Moreover, by \cite[Lemma 2.3]{DS} it holds
$$\widehat{G}^+=\pi(G^+), \;\;\;\widehat{G}^0=\pi(G^0)\;\;\;\mbox{ and }\;\;\;\widehat{G}^-=\pi(G^-),$$
where $\pi:G\rightarrow \widehat{G}$ is the canonical projection. Therefore, the equalities $\pi^{-1}(\pi(G^0))=N^0G^0=G^0$ and Proposition \ref{control}, implies 
	$$\mathrm{Per}(\Sigma_G)\;\;\;\mbox{is a bounded set }\;\;\;\iff\;\;\;\mathrm{Per}\bigl(\Sigma_{\widehat{G}}\bigr)\;\;\;\mbox{is a bounded set},$$
	and our proof is reduced to show the same result for the projected linear system $\Sigma_{\widehat{G}}$. However, by Lemma \ref{lemma1} we obtain $N=N^0G^{+, -}$ and so 
$$\widehat{G}^{+, -}=\pi(G^{+, -})=\pi(N^0G^{+, -})=\pi(N).$$
Therefore, $\widehat{G}^{+, -}$ is a subgroup of $\widehat{G}$. By the first case, we get that $\mathrm{Per}(\Sigma_{\widehat{G}})$ is a bounded set and the result follows.  
\end{proof}

\subsection{Control sets}

Roughly speaking, a control set is a maximal region on which the system in controllable. Such concept appears mainly due to the lack of controllability of many systems. For linear system they have been studied in \cite{DSAyGZPH, DSAy2, DSAyGZ}. In this section we formally define control sets and show that, under some conditions, the interior of such a set coincides with the set of the central periodic points. In particular, that gives us a necessary and sufficient condition to assure boundedness of the control set.

\bigskip

For $x\in G$ the {\it set of reachable points from $x$} is by definition 
$$\OC^+(x):=\{\phi(t, x, u), \;t\geq0, u\in\UC\}.$$
Following \cite[Definition 3.1.2]{FCWK} a nonempty set $\CC\subset G$ is said to be a {\it control set }of $\Sigma_G$ if it is maximal (w.r.t. set inclusion) satisfying
\begin{itemize}
	\item[(i)] For any $g\in\CC$ there exists $u\in\UC$ such that $\phi(\R^+, g, u)\subset\CC$;
	\item[(ii)] For any $x\in\CC$ it holds that $\CC\subset \cl(\OC^+(x))$.
\end{itemize}

In \cite{DSAyGZ} the authors studied control sets of linear systems on Lie groups with finite semisimple center, that is, connected Lie groups $G$ such that the associated semisimple Lie group $G/R$ has finite center, where $R$ is the solvable radical of $G$\footnote{The definition of a Lie group with finite semisimple center in \cite[Definition 3.1]{DSAyGZ}, althought different, is trivially equivalent to the previous one by Malcev's Theorem (see \cite[Theorem 4.3]{ALEB}).}. This condition is not restrictive since it covers solvable and reductive Lie groups, semisimple Lie groups with finite center and any semidirect product among the refereed classes. Moreover, by Proposition 3.3 of \cite{DSAyGZ}, if the central subgroup $G^0$ associated with a linear vector field is compact, then $G/R=\pi(G^0)$ is a compact semisimple Lie group and hence has finite center, where $\pi=G\rightarrow G/R$ is the canonical projection. 

Assume then that $G$ has finite semisimple center and let $\Sigma_G$ a linear system on $G$. If $\OC^+(e)$ is an open subset, there exists a control set $\CC$ of $\Sigma_G$ such that $e\in\inner\CC$ (see \cite[Proposition 4.5.11]{FCWK}. Reciprocally, if $\Sigma_G$ admits a control set $\CC$ such that $e\in\inner\CC$, then $\inner\CC\subset\OC^+(e)$, implying that $\OC^+(e)$ is open (see \cite[Lemma 3.2.13]{FCWK}). Furthermore, if such a control set $\CC$ exists, we have by \cite[Theorem 3.8]{DSAy1} that 
$$G^0\subset \OC^+(e)\cap\OC^-(e)\subset\inner\CC,$$
where $\OC^-(x)$ is the set of reachable points from $x$ associated with the reversed-time system. Also, if $G$ is decomposable, $\CC$ is the only control set with nonempty interior of $\Sigma_G$ (see \cite[Theorem 3.11]{DSAyGZ}).

\begin{remark}
	It is important to remark that the condition \cite[Theorem 3.8]{DSAy1}  is that $e\in\inner\OC^{+}_{\tau_0}(e)$ for some $\tau_0>0$. However, in the restricted case it is equivalent to the openness of $\OC^+(e)$ (see \cite[Lemma 4.5.2]{FCWK}).
\end{remark}

We can now prove our main result relating control sets and central periodic points.

\begin{theorem}
	\label{maincontrol}
	Let $G$ be a connected Lie group with finite semisimple center and $\Sigma_G$ a linear system. If $\Sigma_G$ admits a control set $\CC$ such that $e\in\inner\CC$, then 
	$$\inner\CC=\mathrm{Per}(e, \Sigma_G)=\mathrm{Per}(\Sigma_G).$$
	In particular, $\CC$ is a bounded control set if and only if $G^0$ is a compact subgroup.
\end{theorem}

\begin{proof}
	It certainly holds that $\mathrm{Per}(e, \Sigma_G)\subset\mathrm{Per}(G^0, \Sigma_G)=\mathrm{Per}(\Sigma_G)$. Furthermore, control sets with nonempty interior have the no-return property, that is, if $x\in\CC$ and $\phi(\tau, x, u)\in\CC$, for some $\tau>0$ and $u\in\UC$, then $\phi(t, x, u)\in \CC$ for any $t\in[0, \tau]$ (see \cite[Corollary 1.1]{Ka1}). Hence, since by our previous discussion $G^0\subset\inner\CC$, the no-return property implies that $\mathrm{Per}(\Sigma_G)\subset\inner\CC$. On the other hand, by \cite[Theorem 2.4]{DSAyGZ}, any two points in $\inner\CC$ can be joined by a trajectory of $\Sigma_G$ implying that $\inner\CC\subset \mathrm{Per}(e, \Sigma_G)$ and therefore	 $\inner\CC=\mathrm{Per}(e, \Sigma_G)=\mathrm{Per}(\Sigma_G)$. 
	
	By Theorem \ref{teo}, $G^0$ is a compact subgroup if and only if $\mathrm{Per}(\Sigma_G)$ is a bounded subset of $G$. Since $\inner\CC$ is dense in $\CC$, the second assertion follows.
\end{proof}

\begin{remark}
	Let us notice that by our previous discussion, $G^0$ compact implies that $G$ has finite semisimple center and hence the second assertions of Theorem \ref{maincontrol} holds without such assumption. 
\end{remark}

\subsection{Examples}

In this section we provide two examples. The first example on the Abelian group $\R^2$. The second one, on a connected three-dimensional nilpotent Lie group.
 
 \begin{example}
 	\label{R^2}
 	Let us consider a classical linear system $\Sigma_{\R^2}$ on $\R^2$ given in coordinates by 
 $$\left\{\begin{array}{l}
 \dot{x}(t)=x(t)+u(t)\\
 \dot{y}(t)=-y(t)+u(t)
 \end{array}\right.,$$
 where $u(t)\in [-1, 1]$. Following \cite[Example 3.2.27]{FCWK} the control set of $\Sigma_{\R^2}$ is given by 
 	$$\CC=(-1, 1)\times [-1, 1].$$
 	Since the derivation associated to the drift is $\DC=\left(\begin{array}{cc}
 	1 & 0\\ 0 & -1
 	\end{array}\right)$ we get $(\R^2)^0=\{0\}$ and hence $\mathrm{Per}(\Sigma_{\R^2})=(-1, 1)\times (-1, 1)$.
  \end{example}

\begin{example}
	Let us consider $\R^3$ endowed with the bracket obtained by the relations 
	$$[e_1, e_2]=e_3, \;\;\;[e_1, e_3]=[e_2, e_3]=0,$$
	where $\{e_1, e_2, e_3\}$ is the canonical basis of $\R^3$.
	The {\it Heisenberg group} $\mathbb{H}$ is the Lie group whose subjacent manifold is $\mathbb{R}^3$ and the product is given by 
$$v\cdot w=v+w+\frac{1}{2}[v, w].$$
Consider the central discrete subgroup $\{(0, 0, p), \;p\in\Z\}=\Z$ and the connected Lie group $G=\mathbb{H}/\Z\sim \R^2\times\R/\Z$. The maps
$$[v]\mapsto e_1+e_2+\frac{1}{2}[e_1+e_2, v]\;\;\;\mbox{ and } \;\;\;[v]\mapsto \DC v,$$
define a right-invariant vector field and a linear vector field on $G$, respectively. Here $\DC=\mathrm{diag}(1, -1, 0)$ and $[v]=v +\Z$. An easy computation shows that the linear system $\Sigma_{G}$ built up with these two dynamics reads in coordinates as follows
$$\left\{\begin{array}{l}
\dot{x}(t)=x(t)+u(t)\\
\dot{y}(t)=-y(t)+u(t)\\
\dot{z}(t)=\frac{u(t)}{2}(y(t)-x(t)),
\end{array}\right.$$
where we consider $u(t)\in [-1, 1]$. Note that $G^0=Z(G)=\R/\Z$ is a compact normal subgroup of $G$. Moreover, by the previous equations, it turns out that $\Sigma_G$ and the linear system $\Sigma_{\R^2}$ from Example \ref{R^2} are $\pi$-conjugated, where $\pi:G\rightarrow G/G^0$ is the canonical projection. Since $\ker\pi=G^0$, Proposition \ref{control} implies 
$$\pi(\mathrm{Per}(\Sigma_{G}))=\mathrm{Per}(\Sigma_{\R^2})\;\;\;\mbox{ and }\;\;\;\pi^{-1}\left(\mathrm{Per}(\Sigma_{\R^2})\right)=\mathrm{Per}(\Sigma_{G}).$$
Consequently, 
$$\mathrm{Per}(\Sigma_{G})=\left((-1, 1)^2\times\{0\}\right)+\R/\Z=(-1, 1)^2\times \R/\Z.$$
\end{example}

\end{document}